\documentclass[A4,10pt]{article}

\usepackage[top=3cm, bottom=2cm, left=2cm, right=2cm]{geometry}
\usepackage[dvipdfmx]{graphicx}
\usepackage[dvipdfmx]{color}
\usepackage{amsmath,amsthm}
\usepackage{amssymb, amsmath}
\usepackage{amsfonts}
\usepackage{comment}
\usepackage{ascmac}
\usepackage{color}
\usepackage{cite}
\newtheorem{theo}{Theorem}

\newtheorem{defi}{Definition}[section]
\newtheorem{rem}{Remark}[section]
\makeatletter

\@addtoreset{equation}{section}
\makeatother

\title{On solvability of a time-fractional doubly critical semilinear equation, and its quantitative approach to the non-existence result on the classical counterpart.}
\author{Mizuki Kojima}
\date{}
\begin{document}
\maketitle
\begin{abstract}
	We study a time-fractional semilinear heat equation
	\[
	\partial^{\alpha}_t u -\Delta u = u^{p},\ \ \mbox{in}\ (0,T)\times\mathbb{R}^N,\ \ u(0)=u_0\ge0
	\]
	with $u_0\in L^{1}(\mathbb{R}^N)$ and $p=1+2/N$. Here $\partial_t^{\alpha}$ denotes the Caputo derivative of order $\alpha \in (0,1)$. Since the space $L^1(\mathbb{R}^N)$ is scale critical with $p=1+2/N$, this type of equation is known as a doubly critical problem. It is known that the usual doubly critical equation $\partial_t u-\Delta u=u^p$ does not have nonnegative global-in-time solutions, while the time-fractional problem does. Moreover, there exists a singular initial data which admits no local-in-time solution, while the time-fractional equation is solvable for any $L^{1}(\mathbb{R}^N)$ initial data. In this paper, we deduce a necessary condition imposed on $u_0$ for the existence of a nonnegative solution. Furthermore, we obtain corollaries that describe the collapse of the local and global solvability for the time-fractional equation as $\alpha \rightarrow 1$.
\end{abstract}

\section{Introduction}

\subsection{Known results for the Fujita equation}
The Fujita equation,
\begin{equation}\label{eq. intro Fujita eq}
	\left\{
	\begin{aligned}
	\partial_t u -\Delta u &=u^p\ \ &&\mbox{in}\ (0,T)\times \mathbb{R}^N,\\
	u(0)&=u_0\ge 0\ \ &&\mbox{on}\ \mathbb{R}^N,
	\end{aligned}
	\right.
\end{equation}
has been extensively studied by many mathematicians starting with the pioneering work \cite{Fujita} which states the following.

\begin{itemize}
	\item If $1<p\le p_f:=1+2/N$, then (\ref{eq. intro Fujita eq}) possesses no nonnegative global-in-time solutions.
	\item If $p_f<p$, then (\ref{eq. intro Fujita eq}) has a global-in-time solution for an appropriate initial data.
\end{itemize}
The Fujita exponent $p_f$ is related to the well-posedness of (\ref{eq. intro Fujita eq}) in some Lebesgue spaces. Due to \cite{BreCaz96, Wei80}, the following holds.
	\begin{itemize}
		\item Assume either $q>q_c:=N(p-1)/2$ and $q\ge 1$, or $q=q_c>1$. Then, for all $u_0\in L^{q}(\mathbb{R}^N)$, the problem (\ref{eq. intro Fujita eq}) has a local-in-time solution.
		\item For each $1\le q<q_c$, there is $u_0\in L^{q}(\mathbb{R}^N)$ such that (\ref{eq. intro Fujita eq}) possesses no nonnegative solutions.
	\end{itemize}
It is well known that the exponent $q_c$ is deduced from the scale-invariant property of (\ref{eq. intro Fujita eq}). \textit{The doubly critical case}, where $q_c=1\ \Leftrightarrow p=p_f$ for $u_0\in L^{1}(\mathbb{R}^N)$ is not mentioned in the above statement, and \cite{BreCaz96} suggested some open problems related to this case. Indeed, for such critical problems, some researchers have revealed the nonexistence of positive solutions for certain nonnegative initial data. See for instance \cite{Takahashi16, FujiIoku18, HisaIshige18, Miya10} and references therein.

In the discussion of some critical problems, the singularity of the initial data plays an essential role. Namely, it is important to formulate some necessary conditions imposed on the initial data for the existence of local-in-time solutions. For instance, \cite{BarasPierre85} revealed that if (\ref{eq. intro Fujita eq}) has a nonnegative local solution, then $u_0\ge 0$ must satisfy
\begin{equation}\label{eq. BarasPierre}
	\int_{B(\rho)} u_0(y)dy \le \gamma \left| \log\rho \right|^{-N/2}
\end{equation}
for sufficiently small $\rho>0$. Here, we denote by $B(z;\rho)\subset \mathbb{R}^N$ a closed ball of radius $\rho>0$ centered at $z\in \mathbb{R}^N$ and $B(\rho):= B(0;\rho)$. For the space-fractional equation
\begin{equation}\label{eq. intro space fractional equation}
	\partial_tu + (-\Delta)^{\theta/2} u = u^{p},\ \ u(0)=u_0\ge 0
\end{equation}
with $0<\theta \le 2$ and $p=1+\theta/N$, \cite{HisaIshige18} showed that existence of a nonnegative solution of (\ref{eq. intro space fractional equation}) on $(0,T)$ implies
\begin{equation}\label{eq. HisaIshige}
		\int_{B(z;\sigma)} u_0(y)dy \le \gamma \left[ \log\left( e+ \frac{T^{1/\theta}}{\sigma} \right) \right]^{-N/\theta}
\end{equation}
holds for all $z\in \mathbb{R}^N$ and $0<\sigma <T^{1/\theta}$. Here $\gamma$ depends on $N$ and $\theta$. For the related topics, see also \cite{HisaIshige19}.

\subsection{Time-fractional heat equation, analogies and differences}

Let $0<\alpha <1$. The Caputo derivative of order $\alpha$ is defined by
\begin{equation}
	(\partial_t^{\alpha}f)(t) := \frac{1}{\Gamma(1-\alpha)} \frac{d}{dt}\int_{0}^{t} (t-\tau)^{-\alpha} \left( f(\tau)-f(0) \right)d\tau.
	\end{equation}
This mathematical tool has been proposed to model the anomalous diffusion, which is different from the usual one of materials based on the Brownian motion.

For the past decades, many mathematical results on the time-fractional differential equations have been revealed. For the general theory, see for instance \cite{Bajl01, EK04, LPJ12, MAH17, Mahm02, WCX12, Podl99}. For the topic of regularity and maximum principle, see \cite{ACV16, Guid19JMAA, Guid19MJM} and \cite{Luch09, Luch10, Luch11}. Nonlinear problems with time-fractional derivatives also have been extensively discussed, for example, time-fractional Navier-Stokes equations \cite{CP15}, nonlinear Schr\"{o}dinger equations \cite{ZSL17}, the viscosity solution method for fully nonlinear parabolic equations \cite{TY17}, and Cauchy problems in which the initial condition has a space nonlocality \cite{ZhJi10}. Among these results, we mainly focus on the following time-fractional semilinear equation
\begin{equation}\label{eq. intro alpha Fujita eq}
	\left\{
	\begin{aligned}
	\partial_t^{\alpha} u -\Delta u &=u^p\ \ &&\mbox{in}\ (0,T)\times \mathbb{R}^N,\\
	u(0)&=u_0\ \ &&\mbox{on}\ \mathbb{R}^N,
	\end{aligned}
	\right.
\end{equation}
for which not only the analogy of (\ref{eq. intro Fujita eq}), but also interesting different aspects have been revealed. In \cite{ZhanSun15}, it is studied that (\ref{eq. intro alpha Fujita eq}) possesses no nonnegative global-in-time solutions if $1<p<p_f$, while has a global one if $p>p_f$. For $p=p_f$ however, (\ref{eq. intro alpha Fujita eq}) is solvable globally-in-time for some appropriate initial data. Moreover, \cite{ZLS19, MMS21} showed that (\ref{eq. intro alpha Fujita eq}) is locally-in-time solvable for any $u_0\in L^{1}(\mathbb{R}^N)$ in the doubly critical case, while (\ref{eq. intro Fujita eq}) cannot be solved in this case. Formally speaking, in the critical situation, the solvability of (\ref{eq. intro alpha Fujita eq}) is easier to obtain than (\ref{eq. intro Fujita eq}).

To observe this point, let us introduce analogies and differences between the usual heat equation and the time-fractional one. For the linear heat equation $\partial_t u - \Delta u=f$ with $u(0)=u_0$, the solution is represented by the 
Duhamel formula:
\[
u(t) = e^{t\Delta}u_0 + \int_{0}^{t}e^{(t-\tau)\Delta} f(\tau)d\tau
\]
where $e^{t\Delta}v:=G(t,\cdot)\ast v$ and
\[
G(t,x):= \frac{1}{\left( 4\pi t \right)^{N/2}} \exp\left( -\frac{|x|^{2}}{4t} \right).
\]
On the other hand, the solution of the time-fractional linear equation $\partial_t^{\alpha}u-\Delta u=f$ with $u(0)=u_0$ is written as
\begin{equation}\label{eq. mild sol for linear eq}
u(t) = P_{\alpha}(t)u_0 + \alpha \int_{0}^{t} (t-\tau)^{\alpha-1} S_{\alpha}(t-\tau) f(\tau)d\tau
\end{equation}
where
\[
P_{\alpha}(t):= \int_{0}^{\infty} h_{\alpha}(\theta)e^{t^{\alpha}\theta \Delta}d\theta,\ \ S_{\alpha}(t):= \int_{0}^{\infty} \theta h_{\alpha}(\theta)e^{t^{\alpha}\theta \Delta}d\theta
\]
and $h_{\alpha}$ is a probability density function $h_{\alpha}$ on $(0,\infty)$ defined by
\[
h_{\alpha}:= \frac{1}{\alpha} \theta^{-1-1/\alpha} \psi_{\alpha}\left( \theta^{-1/\alpha} \right)
\]
for
\[
\psi_{\alpha}(\theta):= \frac{1}{\pi} \sum_{k=1}^{\infty} (-1)^{k-1} \theta^{-k\alpha -1} \frac{\Gamma(k\alpha+1)}{k!} \sin(k\pi \alpha).
\]
It is known that $h_{\alpha}$ fulfills $h_{\alpha}\ge 0$ and
\begin{equation}\label{eq. int halpha}
\int_{0}^{\infty} \theta^{\nu} h_{\alpha}(\theta)d\theta = \frac{\Gamma(1+\nu)}{\Gamma(1+\alpha\nu)}\ \ \mbox{for}\ \nu>-1,
\end{equation}
\begin{equation}\label{eq. int halpha&exp}
\int_{0}^{\infty} h_{\alpha}(\theta) e^{z \theta} d\theta =E_{\alpha,1}(z),\ \ \int_{0}^{\infty} \alpha \theta h_{\alpha}(\theta) e^{z\theta}=E_{\alpha,\alpha}(z)\ \ \mbox{for}\ z\in \mathbb{C}.
\end{equation}
Here, we define the Mittag-Leffler functions as follows:
\[
E_{\alpha,\beta}(z):= \sum_{k=0}^{\infty} \frac{z^{k}}{\Gamma(\alpha k +\beta)}.
\]
For details, see \cite{Emi00, Mahm02, MAH17, WCX12, ZhanSun15, ZhJi10}. It should be noted that the operators $P_{\alpha}(t)$ and $S_{\alpha}(t)$ no longer have the semigroup property owing to the weighted function.

The $L^p$-$L^q$ estimate is an important property of the heat semigroup. It states that
\begin{equation}\label{eq. intro LpLq est}
	\| e^{t\Delta} u_0\|_{L^q} \le C t^{-\frac{N}{2}\left( \frac{1}{p}-\frac{1}{q} \right)} \|u_0\|_{L^p}
\end{equation}
for $q\ge p$. Using this for $P_{\alpha}(t)$ and $S_{\alpha}(t)$ yields the analogous estimate
\begin{equation}\label{eq. intro LpLq est for alpha}
\begin{aligned}
\| P_{\alpha}(t) u_0\|_{L^{q}}&\le C t^{-\frac{\alpha N}{2}\left( \frac{1}{p}-\frac{1}{q} \right)} \int_{0}^{\infty} h_{\alpha} (\theta) \theta^{-\frac{N}{2}\left( \frac{1}{p}-\frac{1}{q} \right)}d\theta \|u_0\|_{L^p},\\
\| S_{\alpha}(t) u_0\|_{L^{q}}&\le C t^{-\frac{\alpha N}{2}\left( \frac{1}{p}-\frac{1}{q} \right)} \int_{0}^{\infty} h_{\alpha} (\theta) \theta^{1-\frac{N}{2}\left( \frac{1}{p}-\frac{1}{q} \right)}d\theta \|u_0\|_{L^p}.
\end{aligned}
\end{equation}
We remark that we cannot choose the exponents $q\ge p$ arbitrarily owing to a limitation on the integrability of the weighted function (\ref{eq. int halpha}). This is in some ways an inconvenient property, but we fortunately observe that the time singularity of (\ref{eq. intro LpLq est for alpha}) is milder than that of (\ref{eq. intro LpLq est}). This relaxation enables us to construct the solution of (\ref{eq. intro alpha Fujita eq}) even in the doubly critical situation. For details, see the proof of \cite[Theorem~4.4]{ZhanSun15}, \cite[Theorem~5.1]{ZLS19}.

\begin{table}[h]\label{tab. usual and timefractional}
	\centering
	\begin{tabular}{ccc}
		\hline
		Equation & Global solution & Local solution for $u_0\in L^{1}(\mathbb{R}^N)$ \rule[0mm]{0mm}{5mm}\\
		\hline \hline
		(\ref{eq. intro Fujita eq}) &not exist& $\exists u_0\ge0$ admits no nonnegative sol. \rule[0mm]{0mm}{5mm}\\
		(\ref{eq. intro alpha Fujita eq}) &exist& always exists\rule[0mm]{0mm}{5mm}\\
		\hline
	\end{tabular}
\caption{Comparison of the usual heat equation and the time-fractional one}
\end{table}

As we noted above, in the critical situation, the time-fractional problem (\ref{eq. intro LpLq est for alpha}) exhibits different properties from those of (\ref{eq. intro Fujita eq}). However, if the Caputo derivative is a naturally expanded concept of the usual derivative, then (\ref{eq. intro Fujita eq}) and (\ref{eq. intro alpha Fujita eq}) must be connected in some ways when $\alpha\rightarrow 1$ formally. Therefore, the interest in the asymptotic behavior as $\alpha\rightarrow 1$ spontaneously arises. How does the \textit{solvability} of (\ref{eq. intro alpha Fujita eq}) with $\alpha<1$ fail when $\alpha\rightarrow 1$? In the present paper, we shall show the analogous result of (\ref{eq. BarasPierre}) and (\ref{eq. HisaIshige}) to the time-fractional Fujita equation (\ref{eq. intro alpha Fujita eq}). Using that, we discuss the behavior as $\alpha\rightarrow 1$ and formulate a collapse of the local and global solvability of (\ref{eq. intro alpha Fujita eq}) in the doubly critical situation.

\subsection{Main results}

Taking into account the mild-solution formula (\ref{eq. mild sol for linear eq}), we define a formulate of a solution for (\ref{eq. intro alpha Fujita eq}) in $(0,T)$.

\begin{defi}
	Let $u_0$ be a nonnegative function in $\mathbb{R}^N$ and $T\in (0,\infty]$. We say that a nonnegative measurable function $u$ in $(0,T)\times \mathbb{R}^N$ is a solution of (\ref{eq. intro alpha Fujita eq}) in $(0,T)$ if $u(t,\cdot)$ is integrable for each $t\in (0,T)$ and
	\[
	u(t,x) = P_{\alpha}(t)u_0 + \alpha \int_{0}^{t}(t-\tau)^{\alpha-1}S_{\alpha}(t-\tau)u^p(\tau,x)d\tau
	\]
	is satisfied.
\end{defi}

Firstly, we deduce the analogous results to (\ref{eq. BarasPierre}) and (\ref{eq. HisaIshige}) as follows:

\begin{theo}[Necessary condition to the initial data]\label{theo. main theorem}
	Suppose that (\ref{eq. intro alpha Fujita eq}) with $p=p_f$ has a nonnegative solution in $(0,T)$. Then, there exists a constant $\gamma(\alpha)$ satisfying $\limsup_{\alpha\rightarrow 1} \gamma(\alpha)<\infty$ such that for all $z\in \mathbb{R}^N$ and $0<\rho^{2/\alpha}<T$,
	\begin{equation}\label{eq. alpha ver. HisaIshige}
	\int_{B(z;\rho)} u_0(y)dy\le \gamma (\alpha)\left( \frac{T}{\rho^{2/\alpha}} \right)^{(1-\alpha)N/2} \left[ \log\left(\frac{T}{\rho^{2/\alpha}} \right) \right]^{-N/2}.
	\end{equation}
\end{theo}

\begin{rem}
	{\rm
	Theorem~\ref{theo. main theorem} remains the possibility of the existence of a global-in-time solution. Indeed, if we admit arbitrarily large $T>0$, the right hand of (\ref{eq. alpha ver. HisaIshige}) goes to infinity, while that of (\ref{eq. HisaIshige}) goes to zero.
}
\end{rem}

\begin{rem}
	{\rm
	For (\ref{eq. intro Fujita eq}) with $p\neq p_{f}$, the analogous conditions to (\ref{eq. HisaIshige}) are the following (see \cite{HisaIshige18}):
	\[
	\begin{aligned}
	&\int_{B(z;T^{1/2})} u_0(y) dy \le \gamma T^{N/2 - 1/(p-1)}\ \ \mbox{if}\ 1<p<p_{f},\\
	&\int_{B(z;\sigma)}u_0(y)dy \le \gamma \sigma^{N-2/(p-1)}\ \ \mbox{for}\ 0<\sigma<T^{1/2}\ \mbox{if}\ p>p_{f}.
	\end{aligned}
	\]
	Similar estimates for the time-fractional equation (\ref{eq. intro alpha Fujita eq}) with $p\neq p_f$ were shown in \cite[Theorem~4.4]{ZhanSun15}. They used a weak solution method to deduce the estimates, while \cite{HisaIshige18} and we used a mild solution framework.	
}
\end{rem}

Secondly, we focus on the global-in-time solvability for (\ref{eq. intro alpha Fujita eq}) as $\alpha\rightarrow 1$. We denote by $\mathcal{G}_{\alpha}$ a set of all nonnegative integrable initial data which admit a global-in-time solution:
\[
\mathcal{G}_{\alpha}:= \left\{ 0\le v\in L^{1}(\mathbb{R}^N);\ \text{(\ref{eq. intro alpha Fujita eq}) with }u_0=v\text{ possesses a global-in-time solution.} \right\}.
\]
Theorem~\ref{theo. main theorem} directly deduces that $\mathcal{G}_{\alpha}$ is bounded in $L^{1}(\mathbb{R}^N)$. We infer from Table~\ref{tab. usual and timefractional} that $\mathcal{G}_{\alpha}$ \textit{tends to zero} as $\alpha\rightarrow 1$. Indeed, the following holds.

\begin{theo}[Shrinking of $\mathcal{G}_{\alpha}$]\label{cor. behavior of Galpha}
	We have
	\begin{equation}
	\sup_{v\in \mathcal{G}_{\alpha}}\| v\|_{L^1(\mathbb{R}^N)}\le C (1-\alpha)^{N/2}
	\end{equation}
	near $\alpha=1$. In particular,
	\[
	\lim_{\alpha\rightarrow 1}\sup_{v\in \mathcal{G}_{\alpha}}\| v\|_{L^1(\mathbb{R}^N)}=0.
	\]
\end{theo}

Finally, we discuss the local-in-time solvability of (\ref{eq. intro alpha Fujita eq}) in $L^{1}(\mathbb{R}^N)$. In \cite{Miya10}, it is shown that the following $L^{1}(\mathbb{R}^N)$ initial data
\begin{equation}\label{eq. intro Miyamoto init data}
\mu_{\epsilon}(x) :=|x|^{-N} \left( -\log |x| \right)^{-N/2-1+\epsilon} \chi_{B(1/e)} 
\end{equation}
with $0<\epsilon<N/2$ admits no nonnegative solutions of (\ref{eq. intro Fujita eq}) with $p=p_f$. Let $T_{\alpha}$ be the maximal existence time of (\ref{eq. intro alpha Fujita eq}):
\[
T_{\alpha}:= \sup \left\{ T\in (0,\infty);\ \text{there exists a nonnegative solution in }(0,T) \text{ with }p=p_f. \right\}.
\]
As we remarked (see Table~\ref{tab. usual and timefractional}), (\ref{eq. intro alpha Fujita eq}) is solvable even in the doubly critical situation, hence $T_{\alpha}>0$ for any $u_0\in L^{1}(\mathbb{R}^N)$.

\begin{theo}[Decay of the existence time]\label{cor. behavior of Talpha}
	For the initial data $u_0=\mu_{\epsilon}$, we have
	\begin{equation}\label{eq. estimate of Talpha}
	T_{\alpha}\le \exp \left( - \frac{2\epsilon(2-\alpha)}{N\alpha (1-\alpha)} \right).
	\end{equation}
	In particular, $T_{\alpha}\rightarrow 0$ as $\alpha \rightarrow 1$.
\end{theo}

\begin{rem}
	{\rm 
	Theorem~\ref{cor. behavior of Galpha} means that the global-in-time solvability of (\ref{eq. intro alpha Fujita eq}) for nontrivial initial data fails when $\alpha \rightarrow 1$. Furthermore, Theorem~\ref{cor. behavior of Talpha} suggests that the local-in-time solvability of (\ref{eq. intro alpha Fujita eq}) with $u_{0}=\mu_{\epsilon}$ fails when $\alpha \rightarrow 1$. Therefore, these theorems are regarded as the representation of collapse of the local and global-in-time solvability for (\ref{eq. intro alpha Fujita eq}) as $\alpha \rightarrow 1$.
}
\end{rem}

This paper is organized as follows. In Section~\ref{section. Proof of Thm1}, we prove Theorem~\ref{theo. main theorem}. For the proof, we employ the iteration method used in \cite{HisaIshige18}, but modify the way of calculation due to the property of the time-fractional derivative. In Section~\ref{Aspects}, we prove Theorem~\ref{cor. behavior of Galpha} and Theorem~\ref{cor. behavior of Talpha} by using Theorem~\ref{theo. main theorem}.

\section{Necessary condition to the initial data}\label{section. Proof of Thm1}

In this section, we prove Theorem~\ref{theo. main theorem}. We refer to the iteration method used in \cite{HisaIshige18}, which is based on \cite{GK99, Sugitani75, LN92}.

\begin{proof}[Proof of Theorem~\ref{theo. main theorem}]
	Let $0<\rho^{2/\alpha}<t<T$. We simply denote
	\[
	M:= \int_{B(z;\rho)} u_0(y)dy.
	\]
	By the definition,
	\[
	u(t,x+z) = \int_{\mathbb{R}^N}\int_{0}^{\infty} h_{\alpha}(\theta) G(t^{\alpha} \theta, x+z-y) u_0(y)dy + \alpha \int_{0}^{t} (t-\tau)^{\alpha-1} \left[ S_{\alpha}(t-\tau) u^p(\tau) \right](x+z) d\tau.
	\]
	Multiplying $G(t^{\alpha},x)$ by the both side and integrating with respect to $x$, we obtain
	\[
	\begin{aligned}
	\int_{\mathbb{R}^N}u(t,x+z)G(t^{\alpha},x)dx&= \int_{\mathbb{R}^N}\left[ \int_{0}^{\infty} h_{\alpha}(\theta) \left(\int_{\mathbb{R}^N} G(t^{\alpha}\theta, x+z-y) G(t^{\alpha},x) dx\right)d\theta \right] u_0(y) dy \\
	&\hspace{5ex}+ \alpha \int_{\mathbb{R}^N} G(t^{\alpha}, x) \left(\int_{0}^{t} (t-\tau)^{\alpha-1} \left[ S_{\alpha}(t-\tau) u^p(\tau) \right](x+z) d\tau \right)dx.
	\end{aligned}
	\]
	For the first term, by using the semigroup property of the heat kernel $G$, we deduce
	\[
	\begin{aligned}
	&\int_{\mathbb{R}^N}\left[ \int_{0}^{\infty} h_{\alpha}(\theta) \left(\int_{\mathbb{R}^N} G(t^{\alpha}\theta, x+z-y) G(t^{\alpha},x) dx\right)d\theta \right] u_0(y) dy\\
	&= \int_{\mathbb{R}^N}\left[ \int_{0}^{\infty} h_{\alpha}(\theta) G\left( t^{\alpha}(1+\theta), z-y \right)d\theta \right] u_0(y) dy\\
	&\ge C \left( \int_{0}^{\infty} h_{\alpha}(\theta) \left( 1+ \theta \right)^{-N/2} d\theta \right) t^{-\alpha N/2}\int_{B(z;\rho)} \exp\left( - \frac{|z-y|^{2}}{4t^{\alpha}} \right) u_0(y) dy\\
	&\ge C \left( \int_{0}^{\infty} h_{\alpha}(\theta) \left( 1+ \theta \right)^{-N/2} d\theta \right) t^{-\alpha N/2} M.
	\end{aligned}
	\]
	For the Duhamel term, similarly we use the semigroup property to get
	\[
	\begin{aligned}
	&\int_{\mathbb{R}^N} G(t^{\alpha}, x) \left(\int_{0}^{t} (t-\tau)^{\alpha-1} \left[ S_{\alpha}(t-\tau) u^p(\tau) \right](x+z) d\tau \right)dx\\
	&=  \int_{\mathbb{R}^N} G(t^{\alpha}, x) \left[\int_{0}^{t} (t-\tau)^{\alpha-1} \left\{ \int_{\mathbb{R}^N} \left( \int_{0}^{\infty}\theta h_{\alpha}(\theta) G((t-\tau)^{\alpha}\theta, x+z-y)d\theta \right) u^p(\tau,y) dy \right\} d\tau \right]dx\\
	&= \int_{0}^{t} (t-\tau)^{\alpha-1} \left[ \int_{\mathbb{R}^N} \left\{ \int_{0}^{\infty} \theta h_{\alpha}(\theta) \left( \int_{\mathbb{R}^N} G(t^{\alpha}, x) G((t-\tau)^{\alpha}\theta, x+z-y)dx \right) d\theta\right\} u^p(\tau,y) dy \right] d\tau\\
	&= \int_{0}^{t} (t-\tau)^{\alpha-1} \left\{ \int_{\mathbb{R}^N} \left( \int_{0}^{\infty} \theta h_{\alpha}(\theta) G\left( t^{\alpha}+(t-\tau)^{\alpha}\theta, y-z \right) d\theta \right) u^p(\tau,y)dy \right\} d\tau.
	\end{aligned}
	\]
	The following estimates
	\[
	\begin{aligned}
	&\exp\left( - \frac{|y-z|^{2}}{4t^{\alpha} + 4(t-\tau)^{\alpha}\theta} \right) \ge \exp\left( -\frac{|y-z|^{2}}{4\tau^{\alpha}} \right),\\
	&t^{\alpha}+(t-\tau)^{\alpha}\theta \le t^{\alpha}(1+\theta)= \tau^{\alpha} \left( \frac{t}{\tau} \right)^{\alpha}(1+\theta)
	\end{aligned}
	\]
	and Jensen's inequality provide
	\begin{equation}\label{eq. w(t)}
	\begin{aligned}
	w(t)&\ge C_1\left( \int_{0}^{\infty} h_{\alpha}(\theta) \left( 1+ \theta \right)^{-N/2} d\theta \right) t^{-\alpha N/2} M\\
	&\hspace{5ex} + C_2 \left(  \alpha \int_{0}^{\infty} \theta h_{\alpha}(\theta) (1+\theta)^{-N/2}d\theta \right) t^{\alpha-1} t^{-\alpha N/2} \int_{\rho^{2/\alpha}}^{t} \tau^{\alpha N/2} w^{p}(\tau)d\tau
	\end{aligned}
	\end{equation}
	where
	\[
	w(t):=\int_{\mathbb{R}^N}u(t,x+z)G(t^{\alpha},x)dx
	\]
	which is finite due to the integrability of $x\mapsto u(t,x+z)$. Since $e^{-r\theta} (1+\theta)^{N/2}$ is bounded in $[0,\infty)$ for arbitrarily fixed $r>0$, we use (\ref{eq. int halpha&exp}) to deduce
	\[
	\begin{aligned}
	&C\int_{0}^{\infty} h_{\alpha}(\theta) (1+\theta)^{-N/2} d\theta \ge \int_{0}^{\infty} h_{\alpha}(\theta) e^{-r\theta}d\theta = E_{\alpha,1}(-r)=: r_1(\alpha),\\
	&C\alpha \int_{0}^{\infty} \theta h_{\alpha}(\theta) (1+\theta)^{-N/2} d\theta \ge \alpha \int_{0}^{\infty} \theta h_{\alpha}(\theta) e^{-r\theta}d\theta = E_{\alpha,\alpha}(-r)=:r_2(\alpha).
	\end{aligned}
	\]
	Therefore, (\ref{eq. w(t)}) is rewritten as
	\begin{equation}\label{eq. estimate of w(t)}
		w(t)\ge C_1 r_1(\alpha) M t^{-\alpha N/2} + C_2 r_2(\alpha) t^{\alpha-1}t^{-\alpha N/2} \int_{\rho^{2/\alpha}}^{t} \tau^{\alpha N/2} w^{p}(\tau) d\tau.
	\end{equation}
	We apply the iteration argument in \cite{HisaIshige18} to (\ref{eq. estimate of w(t)}). Let
	\[
	a_1= C_1 r_1(\alpha),\ \ \ a_{k+1} =C_2 r_2(\alpha) a_k^{p} \frac{p-1}{p^{k}-1},\ \ k=1,2,\cdots
	\]
	and
	\[
	\begin{aligned}
	&q_1=0,\ \ \ q_{k+1}=p q_{k} +(\alpha-1)\le 0,\ \ k=1,2,\cdots\\
	&r_1=0,\ \ \ r_{k+1}=p r_{k} +(1-\alpha)\ge 0,\ \ k=1,2,\cdots.
	\end{aligned}
	\]
	We claim that for all $k=1,2,\cdots$,
	\begin{equation}\label{eq. estimate of w wrt k}
		w(t) \ge a_k M^{p^{k-1}} t^{q_k} \left( \rho^{2/\alpha} \right)^{r_k} t^{-\alpha N/2} \left[ \log \left( \frac{t}{\rho^{2/\alpha}} \right) \right]^{\frac{p^{k-1}-1}{p-1}}.
	\end{equation}
	For $k=1$, the estimate (\ref{eq. estimate of w wrt k}) is obviously true. Assume that (\ref{eq. estimate of w wrt k}) is true with $k$. Then, using (\ref{eq. estimate of w(t)}), we obtain
	\[
	\begin{aligned}
	w(t)&\ge C_2r_2(\alpha) t^{\alpha-1} t^{-\alpha N/2} \int_{\rho^{2/\alpha}}^{t} \tau^{\alpha N/2} a_{k}^{p} M^{p^k} \tau^{pq_k} \left( \rho^{2/\alpha} \right)^{pr_k} \tau^{-\alpha N p/2} \left[ \log \left( \frac{\tau}{\rho^{2/\alpha}} \right) \right]^{\frac{p\left(p^{k-1}-1\right)}{p-1}}d\tau\\
	&= a_{k+1} M^{p^k} t^{q_{k+1}} \left( \rho^{2/\alpha} \right)^{pr_{k}} t^{-\alpha N/2} \int_{\rho^{2/\alpha}}^{t} \tau^{-\alpha} \tau^{\alpha-1}\tau^{1-\alpha} \left[ \log \left( \frac{\tau}{\rho^{2/\alpha}} \right) \right]^{\frac{p\left(p^{k-1}-1\right)}{p-1}}d\tau\\
	&\ge  a_{k+1} M^{p^k} t^{q_{k+1}} \left( \rho^{2/\alpha} \right)^{r_{k+1}} t^{-\alpha N/2} \int_{\rho^{2/\alpha}}^{t} \tau^{-1} \left[ \log \left( \frac{\tau}{\rho^{2/\alpha}} \right) \right]^{\frac{p\left(p^{k-1}-1\right)}{p-1}}d\tau\\
	&=  a_{k+1} M^{p^k} t^{q_{k+1}} \left( \rho^{2/\alpha} \right)^{r_{k+1}} t^{-\alpha N/2} \left[ \log \left( \frac{\tau}{\rho^{2/\alpha}} \right) \right]^{\frac{p^{k}-1}{p-1}}.
	\end{aligned}
	\]
	Therefore, (\ref{eq. estimate of w wrt k}) is true for all $k$. Moreover, we say that there exist constants $C_3,\ C_4>0$ such that
	\begin{equation}\label{eq. estimate of ak}
		a_{k} \ge \left[ C_3 \left( \frac{r_1(\alpha)}{r_2(\alpha)} \right)^{C_4} \right]^{p^k}.
	\end{equation}
	Indeed, let $b_k:= -p^{-k}\log a_k$. Then,
	\[
	\begin{aligned}
	b_{k+1}-b_{k} &= p^{-k-1} \log \left( C_2 r_2(\alpha) \frac{p^k-1}{p-1} \right)\\
	&\le p^{-k-1} \left( \log \left( \frac{C_2r_2(\alpha)}{p-1} \right) +\log p^k \right)\\
	&\le p^{-k-1} \left( C k + \log C r_2(\alpha) \right).
	\end{aligned}
	\]
	Therefore,
	\[
	\begin{aligned}
	b_{k+1} &= b_1+ \sum_{j=1}^{k}\left( b_{j+1}-b_{j} \right)\\
	&\le C + C \log\left( \frac{r_2(\alpha)}{r_1(\alpha)} \right).
	\end{aligned}
	\]
	This estimate deduces (\ref{eq. estimate of ak}). Then, we obtain
	\[
	w(t) \ge \left[ C_3 \left( \frac{r_1(\alpha)}{r_2(\alpha)} \right)^{C_4} \right]^{p^k} M^{p^{k-1}} \left( t^{\alpha-1} \right)^{\frac{p^{k-1}}{p-1}} \left( \rho^{2(1-\alpha)/\alpha} \right)^{\frac{p^{k-1}}{p-1}} t^{-\alpha N/2} \left[ \log\left( \frac{t}{\rho^{2/\alpha}} \right) \right]^{\frac{p^{k-1}-1}{p-1}}.
	\]
	Since $w(t)<\infty$, it follows that
	\[
	M\le C \left( \frac{r_2(\alpha)}{r_1(\alpha)} \right)^{C_4/p} \left( \frac{t}{\rho^{2/\alpha}} \right)^{(1-\alpha)N/2} \left[ \log \left( \frac{t}{\rho^{2/\alpha}} \right) \right]^{-N/2}.
	\]
	Since $r_1(\alpha),\ r_2(\alpha)\rightarrow e^{-r}$ as $\alpha\rightarrow 1$, we obtain the desired result.
\end{proof}

\section{Behavior when $\alpha \rightarrow 1$}\label{Aspects}

In this section, we prove Theorem~\ref{cor. behavior of Galpha} and Theorem~\ref{cor. behavior of Talpha}.

\subsection{Shrinking of $\mathcal{G}_{\alpha}$}

Let us prove Theorem~\ref{cor. behavior of Galpha}. Let $\alpha\in (0,1)$ and $\rho$ be arbitrarily fixed. Let $f: (1,\infty)\ni x\mapsto x^{(1-\alpha)N/2} \left( \log x \right)^{-N/2}$. We see that $f$ decreases if $ 1<x\le e^{1/(1-\alpha)}$, while increases if $x\ge e^{1/(1-\alpha)}$. Therefore, $f$ takes the minimum value $f(e^{1/(1-\alpha)})=e^{N/2} (1-\alpha)^{N/2}$. For arbitrary $\rho>0$, we auxiliary take $T:=e^{1/(1-\alpha)}\rho^{2/\alpha}>\rho^{2/\alpha}$. Then, for all $v\in \mathcal{G}_{\alpha}$,
\[
\int_{B(\rho)} v(y)dy \le C(1-\alpha)^{N/2}. 
\]
Taking $\rho\rightarrow \infty$, we botain
\[
\| v\|_{L^1(\mathbb{R}^N)} \le C(1-\alpha)^{N/2}.
\]
Thus, we complete the proof.
\qed

\subsection{Decay of the existence time}
	Now we prove Theorem~\ref{cor. behavior of Talpha} in this subsection. Assume that there exist $T>0$ and a sequence $\{\alpha_k \}$ such that $\alpha_k\rightarrow 1$ as $k\rightarrow \infty$ and $T_{\alpha_k} \ge T$ for all $k$. For $\rho<1/e$, we calculate
	\[
	\int_{B(\rho)} \mu_{\epsilon}(y)dy=C \left[ \log\left( \frac{1}{\rho} \right) \right]^{-N/2+\epsilon}.
	\]
	Let us take $\rho<\min(1,T^{1/2})$. Then, for any $k$, we have $\rho^{2/\alpha_k}<\rho^{2}<T$. Hence,
	\[
	\left[ \log \left( \frac{1}{\rho} \right) \right]^{-N/2+\epsilon}\le C \gamma(\alpha_k) \left( \frac{T}{\rho^{2/\alpha_k}} \right)^{(1-\alpha_k)N/2} \left[ \log\left( \frac{T}{\rho^{2/\alpha_k}} \right) \right]^{-N/2}.
	\]
	Therefore, taking $k\rightarrow \infty$ yields that
	\[
	\left[ \log \left( \frac{1}{\rho}\right) \right]^{-N/2+\epsilon}\le C \left[ \log\left( \frac{T}{\rho^2} \right) \right]^{-N/2}=C\left[ \log T+ 2\log\left( \frac{1}{\rho} \right) \right]^{-N/2}
	\]
	holds for any $\rho\in (0,\rho_*)$ with $\rho_*$ is sufficiently small. Denoting $y:=\log\left(1/\rho \right)$ for simplicity, we obtain
	\[
	\left(\frac{\log T}{y} + 2 \right)^{N/2} y^{\epsilon}\le C.
	\]
	Since $y \rightarrow \infty$ as $\rho\rightarrow 0$, we deduce the contradiction. Let us prove the estimate (\ref{eq. estimate of Talpha}). Let $T(\alpha):= \min\left(1,T_{\alpha}\right)$. We take $\rho_{\alpha}^{2/\alpha-1}=T(\alpha)$ so that $T(\alpha)/\rho^{2/\alpha}_{\alpha} = 1/\rho_{\alpha}>1$. We see that
	\[
	\left( -\log\rho_{\alpha} \right)^{\epsilon} \le C \left( \frac{1}{\rho_{\alpha}} \right)^{(1-\alpha)N/2}\ \Leftrightarrow\ \rho_{\alpha}^{(1-\alpha)N/2} \left( -\log \rho_{\alpha} \right)^{\epsilon}\le C.
	\]
	Replacing $\sigma_{\alpha} :=-\log \rho_{\alpha}$, we deduce
	\[
	\sigma_{\alpha}^{\epsilon}\exp\left(-\frac{(1-\alpha)N\sigma_{\alpha}}{2}\right) \le C.
	\]
	Here, we remark that $\rho_{\alpha}\rightarrow 0$ implies $\sigma_{\alpha}\rightarrow \infty$ as $\alpha \rightarrow 1$. We consider the function $g: [0,\infty)\ni x\mapsto x^{\epsilon} \exp\left(-(1-\alpha)Nx/2\right) $. We see that $g$ increases if $0\le x \le \frac{2\epsilon}{N(1-\alpha)}$, while decreases if $x\ge \frac{2\epsilon}{N(1-\alpha)}$. Furthermore, for each point $x\in [0,\infty)$, $g$ is monotone increasing with respect to $\alpha$. Since $\sigma_{\alpha} \rightarrow \infty$, we conclude that $\frac{2\epsilon}{N(1-\alpha)}\le \sigma_{\alpha}$. Hence,
	\[
	\rho_{\alpha} \le \exp\left( -\frac{2\epsilon}{N(1-\alpha)} \right).
	\]
	Thus,
	\[
	T(\alpha) = \rho_{\alpha}^{2/\alpha-1} \le \exp\left( -\frac{2\epsilon(2-\alpha)}{N\alpha(1-\alpha)} \right).
	\]
\qed

\section*{Acknowledgment}
I would like to thank my instructor Michiaki Onodera, associate professor of the Department of Mathematics, School of Science, Tokyo Institute of Technology, for his helpful comments on this study.

\bibliographystyle{plain}
\bibliography{dcref.bib}

\end{document}